\documentclass [11pt] {article}
\usepackage{amsfonts}
\usepackage{amssymb}
\usepackage{times}
\usepackage{graphicx}

\newtheorem{lemma}{Lemma}[section]
\newtheorem{corollary}{Corollary}[section]
\newtheorem{theorem}{Theorem}

\newtheorem{proposition}{Proposition}[section]
\newtheorem{definition}{Definition}[section]
\def\supp{{\rm supp}}

\def\<{\langle}
\def\>{\rangle}

\def\to{\rightarrow}

\def\bX{{\bf X}}

\def\bA{{\bf A}}

\begin{document}

\title{Limits of functions on groups}
\author{{\sc Bal\'azs Szegedy}}

\maketitle

\abstract{Our goal is to develop a limit approach for a class of problems in additive combinatorics that is analogous to the limit theory of dense graph sequences. We introduce metric, convergence and limit objects for functions on groups and for measurable functions on compact abelian groups. As an application we find exact minimizers for densities of linear configurations of complexity $1$.}

\section{Introduction}

The so-called graph limit theory (see \cite{LSz}, \cite{LSz2}, \cite{BCLSV}, \cite{L}) gives an analytic approach to a large class of problems in graph theory. A very active field of applications is extremal graph theory where, roughly speaking, the goal is to find the maximal (or minimal) possible value of a graph parameter in a given family of graphs and to study the structure of graphs attaining the extremal value. A classical example is Tur\'an's theorem which implies that a triangle free graph $H$ on $2n$ vertices maximizes the number of edges if $H$ is the complete bipartite graph with equal color classes. Another example is given by the Chung-Graham-Wilson theorem \cite{CGW}. If we wish to minimize the density of the four cycles in a graph $H$ with edge density $1/2$ then $H$ has to be sufficiently quasi random. However the perfect minimum of the problem (that is $1/16$) can not be attained by any finite graph but one can get arbitrarily close to it. Such problems justify graph limit theory where in an appropriate completion of the set of graphs the optimum can always be attained if the extremal problem satisfies a certain continuity property. Furthermore one can use variational principles at the exact maximum or minimum bringing the tools of differential calculus into graph theory.

Extremal graph (and hypergraph) theory has a close connection to additive combinatorics. It is well known that the triangle removal lemma by Szemer\'edi and Ruzsa implies Roth's theorem on three term arithmetic progressions. The proof relies on an encoding of an integer sequence (or a subset in an abelian group) by a graph that is rather similar to a Cayley graph. Such representations of additive problems in graph theory hint at a limit theory for subsets in abelian groups that is closely connected to graph limit theory. This new limit theory, that is actually a limit theory for functions on abelian groups, was initiated by the author in \cite{Sz2} \cite{Sz3} and \cite{Sz} in a rather general form. It turns out that there is a hiararchy of limit notions corresponding to $k$-th order Fourier analysis where the limit notion gets finer as $k$ is increasing and the limit objects get more complicated. The focus of this paper is the linear case $k=1$ that was called ``harmonic analytic limit'' in \cite{Sz2}. This case is interesting on its own right, covers numerous important questions and is illustrative for the more general limit concept. 

We introduce metric, convergence and limit objects for subsets in abelian groups. More generally, since subsets can be represented by their characteristic functions, we study the convergence of functions on abelian groups. This extends the range of possible applications of our approach to problems outside additive combinatorics. 

In the first part of the paper we study a metric $\hat{d}$ and related convergence notion for $l^2$ functions on discrete (not necessarily commutative) groups. It is important that the metric $\hat{d}$ allows us to compare two functions defined on different groups. In chapter \ref{chapcomplim} we introduce a distance $d$ for measurable functions $f\in L^2(A_1),g\in L^2(A_2)$ defined on compact ablelian groups $A_1,A_2$ such that $d(f,g):=\hat{d}(\hat{f},\hat{g})$ where $\hat{f}$ and $\hat{g}$ denote the Fourier transforms of $f$ and $g$. In additive combinatorics, we can use the distance $d$ to compare subsets in finite abelian groups in the following way. If $S_1\subseteq A_1$ and $S_2\subseteq A_2$ are subsets in finite abelian groups $A_1$ and $A_2$ then their distance is $d(1_{S_1},1_{S_2})$. This allows us to talk about convergent sequences of subsets in a sequence of abelian groups. 

A crucial property of the metric $d$ (see theorem \ref{convcl}) is that it puts a compact topology on the set of all pairs $(f,A)$ where $A$ is a compact abelian group and $f$ is a measurable function on $A$ with values in a fixed compact convex set $K\subset\mathbb{C}$. As a consequence we have that any sequence of subsets $\{S_i\subseteq A_i\}_{i=1}^\infty$ in finite abelian groups $A_i$ has a convergent sub-sequence with limit object which is a measurable function of the form $f:A\rightarrow [0,1]$ where $A$ is some compact abelian group. This result is analogous to graph limit theory where graph sequences always have convergent subsequences with limit object which is a symmetric measurable function of the form $W:[0,1]^2\rightarrow [0,1]$. 

The success of a limit theory depends on how many interesting parameters are continuous with respect to the convergence notion. The parameters that are most interesting in additive combinatorics are densities of linear configurations. A linear configuration is given by a finite set of linear forms i.e. homogeneous linear multivariate polynomials over $\mathbb{Z}$. For example a $3$ term arithmetic progression is given by the linear forms $a,~a+b,~a+2b$. If $f$ is a bounded measurable function on a compact abelian group $A$ then we can compute the density of $3$-term arithmetic progressions in $f$ as the expected value $\mathbb{E}_{a,b\in A}(f(a)f(a+b)f(a+2b))$ according to the normalized Haar measure on $A$. This density concept can be generalized to an arbitrary linear configuration $\mathcal{L}=\{L_1,L_2,\dots,L_k\}$ and the density of $\mathcal{L}$ in $f$ is denoted by $t(\mathcal{L},f)$ (see formula (\ref{confdens}) and the following sentence.). 
Gowers and Wolf introduced a complexity notion \cite{GW} for linear configurations called {\it true complexity} (see definition \ref{truecomp} in this paper). A useful upper bound for the true complexity is the so-called Cauchy-Schwarz complexity developed by Green and Tao in \cite{GT}. 

We prove the following fact (for precise formulation see theorem \ref{denscont}).

\medskip

\noindent {\bf Theorem:}~~{\it If $\mathcal{L}$ ha true complexity at most $1$ then the density function of $\mathcal{L}$ is continuous in the metric $d$.}

\medskip

Examples for linear configurations of complexity $1$ include the $3$-term arithmetic progression \cite{GT}, the parallelogram $a,~a+b,~a+c,~a+b+c$,  and the system $\mathcal{L}_H:=\{x_i+x_j~:~(i,j)\in E(H)\}$ where $H$ is an arbitrary finite graph on $\{1,2,\dots,n\}$.  The last example gives a close connection with graph limit theory. The density of $\mathcal{L}_H$ in $f\in L^\infty(A)$ is equal to the density of the graph $H$ in the symmetric kernel $W:A\times A\rightarrow\mathbb{C}$ defined by $W(x,y)=f(x+y)$. Note that if $f$ has values in $[0,1]$ then $W$ is a graphon in the graph limit language. We will elaborate on this connection in chapter \ref{graphlim}

Let $\mathcal{L}$ be an arbitrary linear configuration. For $0\leq\delta\leq 1$ and $n\in\mathbb{N}$ let $\rho(\delta,n,\mathcal{L})$ denote the minimal possible density of $\mathcal{L}$ in subsets of $\mathbb{Z}_n$ of size at least $\delta n$. Let $\rho(\delta,\mathcal{L}):=\liminf_{p\to\infty}\rho(\delta,p,\mathcal{L})$ where $p$ runs through the prime numbers. A result by Candela and Sisask implies that the $\liminf$ can be relaced by $\lim$ in the definition of $\rho(\delta,\mathcal{L})$. Note that Roth's theorem is equivalent with the fact that $\rho(\delta,\mathcal{L})>0$ if $\delta>0$ and $\mathcal{L}=\{a,~a+b,~a+2b\}$. 

\medskip

%\begin{theorem} For every $0\leq \delta\leq 1$ there is a compact abelian grop $A$ with torsion free dual group and measurable function $f_\delta:A\rightarrow [0,1]$ such that $\mathbb{E}_{x\in A}(f_\delta(x))=\delta$ and $$\rho(\delta)=\mathbb{E}_{a,b\in A}(f_\delta(a)f_\delta(a+b)f_\delta(a+2b)).$$
%\end{theorem}

\begin{theorem}\label{rothext} Let $\mathcal{L}$ be a linear configuration of true complexity at most $1$. For every $0\leq \delta\leq 1$ we have that $$\rho(\delta,\mathcal{L})=\min_f(t(\mathcal{L},f))$$ where $f$ runs through all measurable functions of the form $f:A\rightarrow [0,1]$ with $\mathbb{E}(f)=\delta$ on compact abelian groups $A$ with torsion-free dual groups. 
\end{theorem}

\medskip

We emphasize that in theorem \ref{rothext} we obtain $\rho(\delta,\mathcal{L})$ as an actual minimum and thus there is some function $f_{\delta,\mathcal{L}}$ realizing the value $\rho(\delta)$. If for example $\mathcal{L}=\{a,~a+b,~a+2b\}$ then it is easy to deduce Roth's theorem by using Lebesgue density theorem for a sufficiently precise approximation of $f_{\delta,\mathcal{L}}$ by its projection to a large enough finite dimensional factor group of $A$. One gets that $f_{\delta,\mathcal{L}}$ has positive $3$-term arithmetic progression density if $\delta>0$ and thus $\rho(\delta)>0$ holds. It would be very interesting to find the explicit form of a minimizer $f_{\delta,\mathcal{L}}$ for every $\delta$ or even to obtain any information on $f_{\delta,\mathcal{L}}$ like on which abelian group it is defined? 

\medskip

It is important to mention that our convergence notion behaves quite differently from usual convergence notions in functional analysis. There is an example for a convergent sequence of functions, all of them defined on the circle (complex unit circle with multiplication or equivalently the quotient group $\mathbb{R}/\mathbb{Z}$), but the limit object exists only on the torus.

\medskip

In the proofs we will extensively use ultra limit methods. Ultralimt methods in graph and hypergraph regularization and limit theory were first introduced in \cite{ESz}. There are two different reasons to use these methods. One is that they seem to help to get rid of a great deal of technical difficulties and provide cleaner proofs for most of our statements. The other reason is that they point to an interesting connection between ergodic theory and our limit theory. The ultra product $\bA$ of compact abelian groups $\{A_i\}_{i=1}^\infty$ behaves as a measure preserving system. Our limit concept can easily be explained through a factor $\mathcal{F}(\bA)$ of $\bA$ which is a variant of the so called Kronecker factor.

\section{A limit notion for functions on discrete groups}\label{disclim}

For an arbitrary group $G$ we denote by $l^2(G)$ the Hilbert space of all functions $f:G\rightarrow\mathbb{C}$ such that $\|f\|_2^2=\sum_{g\in G}|f(g)|^2\leq\infty$.
If $f\in l^2(G)$ and $\epsilon\geq 0$ then we denote by $\supp_\epsilon(f)$ the set $\{g:g\in G,|f(g)|> \epsilon\}|$ In particular $\supp(f):=\supp_0(f)$ is the support of $f$. Not that if $\epsilon>0$ then $|\supp_\epsilon(f)|\leq\|f\|_2^2/\epsilon^2$ and thus $\supp(f)$ is a countable (potentially finite) set. We denote by $\langle f\rangle$ the subgroup of $G$ generated by $\supp(f)$. It is clear that $\langle f\rangle$ is a countable (potentially finite) group.

Two functions $f_1\in l^2(G_1)$ and $f_2\in l^2(G_2)$ are called isomorphic if there is a group isomorphism $\alpha:\langle f_1\rangle\rightarrow \langle f_2\rangle$ such that $f_1=f_2\circ\alpha$. Let us denote by $\mathcal{M}$ the isomorphism classes of $l^2$ functions on groups. Our goal is to define a metric space structure on $\mathcal{M}$.  We will need the next group theoretic notion.

\begin{definition} Let $G_1$ and $G_2$ be groups. A partial isomorphism of weight $n$ is a bijection $\phi:S_1\rightarrow S_2$ between two subsets $S_1\subseteq G_1,S_2\subseteq G_2$ such that  $g_1^{\alpha_1}g_2^{\alpha_2}\dots g_n^{\alpha_n}=1$ holds if and only if $\phi(g_1)^{\alpha_1}\phi(g_2)^{\alpha_2}\dots \phi(g_n)^{\alpha_n}=1$ for every sequence $g_i\in S_1,\alpha_i\in\{-1,0,1\}$ with $1\leq i\leq n$.
\end{definition}

\begin{definition} Let $f_1\in l^2(G_1)$ and $f_2\in l^2(G_2)$. An $\epsilon$-isomorphism between $f_1$ and $f_2$ is a partial isomorphism $\phi:S_1\rightarrow S_2$ of weight $\lceil 1/\epsilon\rceil$ between sets with $\supp_\epsilon(f_1)\subseteq S_1\subseteq G_1$ and $\supp_\epsilon(f_2)\subseteq S_2\subseteq G_2$ such that $|f_1(g)-f_2(\phi(g))|\leq\epsilon$ holds for every $g\in S_1$. We define $\hat{d}(f_1,f_2)$ as the infimum of all $\epsilon$'s such that there is an $\epsilon$-isomorphism between $f_1$ and $f_2$.
\end{definition}

\begin{proposition} The function $\hat{d}$ is a metric on $\mathcal{M}$.
\end{proposition}

\begin{proof} First we show that $\hat{d}(f_1,f_2)=0$ if and only if $f_1$ and $f_2$ are isomorphic. If $f_1$ is isomorphic to $f_2$ then it is clear that $d(f_1,f_2)=0$. For the other direction assume w.l.o.g. that $\|f_2\|_2\leq \|f_1\|_2$. Let $\alpha_n:S_{1,n}\rightarrow S_{2,n}$ be an $1/n$-isomorphism between $f_1$ to $f_2$ for every $n$.  Clearly, for every element $g\in\supp(f_1)$ there are finitely many possible elements in the sequence $\{\alpha_n(g)\}_{n=1}^\infty$ since $\lim_{n\rightarrow\infty}f_2(\alpha_n(g))=f_1(g)$ and there are finitely many elements $h$ in $G_2$ on which $f_2(h)>f_1(g)/2$.
Using that the support of $f_1$ is countable we obtain that there is a subsequence $\{\beta_n\}$ of $\{\alpha_n\}$ such that the sequences $\{\beta_n(g)\}$ stabilize (become constant) after finitely many steps for every $g$ with $f_1(g)>0$. This defines a map $\beta=\lim\beta_n$ from $\supp(f_1)$ to $\supp(f_2)$. It is clear that $\beta$ extend to an injective homomorphism from $\langle f_1\rangle$ to $\langle f_2\rangle$ and it satisfies $f_2(\beta(g))=f_1(g)$ for every $g\in\langle f_1\rangle$. Using $\|f_2\|_2\leq\|f_1\|_2$ it follows that every element in $\supp(f_2)$ is in the image of $\beta$ and so $\beta$ is a value preserving isomorphism between $\langle f_1\rangle$ and $\langle f_2\rangle$.

It remains to check the triangle inequality for the metric $d$. Assume that $\alpha:S_1\rightarrow S_2$ is an $\epsilon$ isomorphism between $f_1$ and $f_2$ and assume that $\beta:S'_2\rightarrow S_3$ is an $\epsilon'$ isomorphism between $f_2$ and $f_3$. Without loss of generality we can assume (by reversing arrows if necessary) that $\epsilon'\geq\epsilon$. We have the following inclusions: $$\beta^{-1}(\supp_{\epsilon'+\epsilon}(f_3))\subseteq\beta^{-1}(\supp_{\epsilon'}(f_3))\subseteq\beta^{-1}(S_3)= S_2',$$  $$\beta^{-1}(\supp_{\epsilon'+\epsilon}(f_3))\subseteq\supp_\epsilon(f_2)\subseteq S_2,$$ $$\alpha(\supp_{\epsilon'+\epsilon}(f_1))\subseteq\supp_{\epsilon'}(f_2)\subseteq S_2\cap S_2'.$$
Let $T_2=\beta^{-1}(\supp_{\epsilon'+\epsilon}(f_3))\cup\supp_{\epsilon'}(f_2)$ (note that $T_2\subseteq S_2\cap S_2'$) and let $T_1=\alpha^{-1}(T_2)$, $T_3=\beta(T_2)$. We have that $\supp_{\epsilon'+\epsilon}(f_1)\subseteq T_1$ and $\supp_{\epsilon'+\epsilon}(f_3)\subseteq T_3$. Let $\gamma:T_1\rightarrow T_3$ be the restriction of $\beta\circ\alpha$ to $T_1$. To complete the proof of the triangle inequality we show that $\gamma$ is an $\epsilon'+\epsilon$ isomorphism. We have that $\gamma$ is a bijection and that $|f_1(g_1)-f_3(\gamma(g_1))|\leq\epsilon'+\epsilon$ holds for every $g\in T_1$. It remains to check that $\gamma$ is a partial isomorphism of weight $\lceil 1/(\epsilon'+\epsilon)\rceil$. This follows form the fact that the composition of a partial isomorphism of weight $n$ and a partial isomorphism of weight $m$ is a partial isomorphism of weight $\min(n,m)$. However the minimum of $\lceil 1/\epsilon\rceil$ and $\lceil 1/\epsilon'\rceil$ is at least $\lceil 1/(\epsilon'+\epsilon)\rceil$. 
\end{proof}

\medskip

\begin{lemma}\label{limisab} Assume that a sequence $\{f_i\}_{i=1}^\infty$ of $l^2$ functions on abelian groups converge in $\hat{d}$ to $f\in l^2(G)$ then $\langle f\rangle$ is also abelian.
\end{lemma}

\begin{proof} Let $g_1,g_2\in\supp(f)$ be two elements. Let $\epsilon=\min(f(g_1)/2,f(g_2)/2,1/4)$. Then by convergence of $f_i$ there is an index $i$ such that there is an $\epsilon$-isomorphism $\phi$ between $f$ and $f_i$. Since $g_1,g_2\in\supp_\epsilon f$ we have that $\phi$ is defined on $g_1,g_2$ and $\phi(g_1)\phi(g_2)\phi(g_1)^{-1}\phi(g_2)^{-1}=1$ implies that $g_1g_2g_1^{-1}g_2^{-1}=1$ because $\epsilon<1/4$.  
\end{proof} 

\medskip

For every real number $a>0$ let $\mathcal{M}_a$ denote the subset of $\mathcal{M}$ consisting of equivalence classes of functions $f\in l^2(G)$ with $\|f\|_2\leq a$.

\begin{proposition}\label{metcomp} The metric space $(\mathcal{M}_a,\hat{d})$ is compact for every $a>0$.
\end{proposition}

Let $F_r$ denote the free group in $r$ generators. We will need the next lemma.

\begin{lemma}\label{bsconv} Assume that $\{G_n\}_{n=1}^\infty$ is a sequence of groups and for every $n$ we have a sequence of elements $\{g_{n,i}\}_{i=1}^\infty$ in $G_n$. Then there is a sequence of elements $\{g_i\}_{i=1}^\infty$ in some group $G$ and a set $S\subseteq\mathbb{N}$ such that for every $r\in\mathbb{N}$ and word $w\in F_r$ there is a natural number $N_w$ such that if $k\in S$ and $k>N_w$ then $w(g_{k,1},g_{k,2},\dots,g_{k,r})=1$ if and only if $w(g_1,g_2,\dots,g_r)=1$. 
\end{lemma}

\begin{proof} Let $\{w_1\}_{i=1}^\infty$ be an arbitrary ordering of the words in $\cup_{r=1}^\infty F_r$ with $w_i\in F_{r_i}$.
We construct a sequence of infinite subsets $S_i\subseteq\mathbb{N}$ in a recursive way. Assume that $S_0=\mathbb{N}$. If $S_{i-1}$ is already constructed then we construct $S_i$ in a way that $S_i$ is an infinite subset in $S_{i-1}$ and either $w_i(g_{s,1},g_{s,2},\dots,g_{s,r_i})=1$ holds for every $s\in S_i$ or $w_i(g_{s,1},g_{s,2},\dots,g_{s,r_i})\neq 1$ holds for every $s\in S_i$. This can be clearly achieved since $S_{i-1}$ is infinite. We then chose a sequence $\{s_i\}_{i=1}^\infty$ such that $s_i\in S_i$ and $s_i<s_j$ hold for every pair $i<j$. We obtain for $\{s_i\}_{i=1}^\infty$ that for every $r\in\mathbb{N}$ and word $w\in F_r$ either $w(g_{s_i,1},g_{s_i,2},\dots,g_{s_i,r})=1$ holds with finitely many exceptions or $w_r(g_{s_i,1},g_{s_i,2},\dots,g_{s_i,r})\neq 1$ holds with finitely many exceptions. Let $W$ denotes the collection of words for which the first case holds. Let $G$ be the group with generators $\{g_i\}_{i=1}^\infty$ and relations $\{w(g_1,g_2,\dots,g_r)=1|r\in \mathbb{N}, w\in F_r\cap W\}$.  It is clear form the construction of $W$ that every relation that $G$ satisfies in its generators is already listed in $W$. This follows from the fact that if a word $w$ is not in $W$ then for an arbitrary finite subset $W'$ in $W$ there is a witness among the groups $G_{s_i}$ in which $w$ does not hold but all words in $W'$ hold. Now we have that $S=\{s_i\}_{i=1}^\infty$ and $G$ with $\{g_i\}_{i=1}^\infty$ satisfies the lemma. 
\end{proof}

\medskip

\noindent{\it Proof of proposition \ref{metcomp}.}~~   Let $\{f_n:G_n\rightarrow\mathbb{C}\}_{n=1}^\infty$ be a sequence of functions of $l^2$ norm at most $a$.
For every $n$ let $\{g_{n,i}\}_{i=1}^\infty$ be an ordering of the elements in $\supp(f_n)$ is such a way that $f_n(g_{n,i})\geq f_n(g_{n,j})$ whenever $i<j$. (if $f_n$ is defined on a finite group then, to make the list infinite, we can extend it to an infinite group containing $G_n$ with $0$ values outside $G_n$.)
Let $S\subseteq\mathbb{N}$, $G$ and $\{g_i\}_{i=1}^\infty$ be chosen for the sequences $\{g_{n,i}\}_{i=1}^\infty$ according to lemma \ref{bsconv}. 
Let $S'\subseteq S$ be an infinite subset of $S$ such  $a_i:=\lim_{n\rightarrow\infty,n\in S'}f_n(g_{n,i})$ exists for every $i\in\mathbb{N}$. Now we define the function $f:G\rightarrow\mathbb{C}$ such that $f(g_i)=a_i$ inside the set $\{g_i\}_{i=1}^\infty$ and $f(g)=0$ for the rest of the elements. It is clear that $f$ is well defined since $g_{n,i}\neq g_{n,j}$ holds for every $n$ if $i\neq j$ and thus $g_i\neq g_j$. It is clear that $\|f\|_2\leq\liminf_{n\in S'}\|f_n\|_2$ and thus $\|f\|_2\leq a$. 

To create an $\epsilon$-isomorphism between $f$ and $f_n$  (if $n\in S'$ is big enough) we consider the sets $T_n=\{g_{n,i}:i\leq a^2/\epsilon^2\}$ and the set $T=\{g_i:i\leq a^2/\epsilon^2\}$. Let $\alpha_n:T_n\rightarrow T$ be the bijection defined by $\alpha_n(g_{n,i})=g_i$. It is clear that $\supp_\epsilon(f_n)\subseteq T_n$ holds for every $n$ and that $\supp_\epsilon(f)\subseteq S$. The construction guarantees that $|f_n(g)-f(\alpha_n(g)|\leq\epsilon$ holds if $n\in S'$ is big enough. Furthermore the property given by lemma \ref{bsconv} shows that $\alpha_n$ is a partial isomorphism of weight $m$ for an arbitrary $m\in \mathbb{N}$ if $n\in S'$ is big enough. This completes the proof.

\medskip

\section{Convergence notions on compact Abelian groups}\label{chapcomplim}

Compact abelian groups in this paper will be assumed to be second countable. In this case the dual group is always countable. For a compact abelian group $G$ we denote by $L^2(G)$ the Hilbert space of Borel measurable complex valued functions  $f$ on $G$ with $\int |f|^2~d\mu\leq\infty$ where $\mu$ is the normalized Haar measure.

Let $f_1\in L^2(G_1)$ and $f_2\in L^2(G_2)$ be functions on the compact abelian groups $G_1$ and $G_2$.
We say that $f_1,f_2$ are isomorphic if there is a third function $f_3\in L^2(G_3)$ and continuous epimorphisms $\alpha_i:G_i\rightarrow G_3$ for $i=1,2$ such that $f_3(\alpha_i(g))=f_i(g)$ holds for almost every $g$ with respect to the Haar measure in $G_i$.

For a function $f\in L^2(G)$ on a compact abelian group we denote by $\hat{f}:\hat{G}\rightarrow\mathbb{C}$ the Fourier transform of $f$ where the discrete group $\hat{G}$ is the dual of $G$. It is clear that $f_1\in L^2(G_1)$ is isomorphic to $f_2\in L^2(G_2)$ if and only if $\hat{f_1}$ is isomorphic to $\hat{f_2}$ in the sense of chapter \ref{disclim}.

Let $\mathcal{H}$ denote the set of isomorphism classes of Borel measurable $L^2$ functions on compact Abelian groups. We introduce the distance $d$ on $\mathcal{H}$ by $d(f_1,f_2):=\hat{d}(\hat{f_1},\hat{f_2})$. The metric $d$ induces a convergence notion on $\mathcal{H}$. If we say $\{f_i\}_{i=1}^\infty$ is convergent then we mean convergence in $d$ if not stated explicitly in which other meaning it is convergent. Let $\mathcal{H}_a$ denote the set of functions in $\mathcal{H}$ with $L^2$-norm at most $a$.  Using the fact that Fourier transform preserves the $L^2$-norm we have by lemma \ref{limisab} and proposition \ref{metcomp} the following statement. 

\begin{proposition} $(\mathcal{H}_a,d)$ is a compact metric space for every $a>0$.
\end{proposition}

For a set $K\subseteq\mathbb{C}$ let $\mathcal{H}(K)$ denote the set of functions in $\mathcal{H}$ which take  values in $K$. We will prove the next theorem.
 
 \begin{theorem}\label{convcl} If $K\subseteq\mathbb{C}$ is a compact convex set then $(\mathcal{H}(K),d)$ is a compact metric space.
 \end{theorem}
 
 \begin{corollary}\label{setconv} IF $\{f_i\}_{i=1}^\infty$ is a sequence of $\{0,1\}$ valued functions in $\mathcal{H}$ converging to $f$ in the metric $d$ then the values of $f$ are in the interval $[0,1]$.
 \end{corollary}

Theorem \ref{convcl} is somewhat surprising. The metric $d$ is given in terms of Fourier transforms however it is not trivial to relate the set of values of a function to the properties of its Fourier transform. The condition that $K$ is convex turns out to be necessary in theorem \ref{convcl}. Corollary \ref{setconv} is useful when we study limits of sets in abelian groups by the limits of their characteristic functions. We give the proof of theorem \ref{convcl} in a later chapter.

We say that a sequence $\{f_i\}_{i=1}^\infty$ in $\mathcal{H}$ is {\it tightly convergent} if it converges in $d$ and the limit $f$ satisfies $\lim_{i\rightarrow\infty}\|f_i\|_2=\|f\|_2$. Tight convergence can be metrized by the distance $$d'(f_1,f_2):=d(f_1,f_2)+|\|f_1\|_2-\|f_2\|_2|.$$
Convergence in $d'$ is stronger than convergence in $d$ and it has stronger consequences. To formulate our result we need the following notation.
For a measurable function $f$ on a compact abelian group $A$ we denote by $\mu_f$ the probability distribution of $f(x)$ where $x$ is chosen randomly from $A$ according to the Haar measure. The measure $\mu_f$ is a Borel probability distribution on $\mathbb{C}$.

\begin{theorem}\label{weakcl} Let $\{f_i\}_{i=1}^\infty$ be a sequence of uniformly bounded functions in $\mathcal{H}$ converging to $f$ in $d'$. Then $\mu_{f_i}$ converges to $\mu_f$ in the weak topology of measures.
\end{theorem}  

Note that the above theorem is not true for convergence in $d$.
A trivial example for a tightly convergent sequence is an $L^2$-convergent sequence of functions on a fixed compact abelian group $A$.
However there are more interesting examples. We finish this chapter with an example which shows that a sequence of $L^2$ functions on the circle group $\mathbb{R}/\mathbb{Z}$ can have a limit (even in $d'$) which can not be defined on the circle group. The limit object exists on the torus.
Let $f_n(x)=e^{2i\pi x}+e^{2in\pi x}$ defined on $\mathbb{R}/\mathbb{Z}$ for $n\in\mathbb{N}$.
It is easy to see that $f_n$ is convergent and the limit is the function $f=e^{2i\pi x}+e^{2i\pi y}$ on the torus $\mathbb{R}/\mathbb{Z}\times\mathbb{R}/\mathbb{Z}$. Note that the sequence $f_n$ is tightly convergent since $\|f_n\|_2=\|f\|_2=\sqrt{2}$.

\section{Densities of linear configurations in functions on Abelian groups}

A linear form is a homogeneous linear multivariate polynomial with coefficients in $\mathbb{Z}$. If $L=\lambda_1x_1+\lambda_2x_2+\dots+\lambda_nx_n$ is a linear form then we can evaluate it in an arbitrary abelian group $A$ by giving values from $A$ to the variables $x_i$ and thus it becomes a function of the form $L:A^n\rightarrow  A$. A system $L_1,L_2,\dots,L_k$ of linear forms determines a type of linear configuration. An example for a linear configuration is the $3$-term arithmetic progression which is encoded by the linear forms $x_1,~x_1+x_2,~x_1+2x_2$. Assume that $A$ is a compact abelian group and $\mathcal{F}=\{f_i\}_{i=1}^k$ is a system of bounded measurable functions in $L^\infty(A)$. Assume furthermore that  $\mathcal{L}=\{L_1,L_2,\dots,L_k\}$ is a sytem of linear forms in $\mathbb{Z}(x_1,x_2,\dots,x_n)$. Then it is usual to define the density of the configuration $\mathcal{L}$ in $\mathcal{F}$ by the formula

\begin{equation}\label{confdens}
t(\mathcal{L},\mathcal{F}):=\mathbb{E}_{x_1,x_2,\dots,x_n\in A}\prod_{i=1}^k f_i(L_i(x_1,x_2,\dots,x_n)).
\end{equation}

If $f_i=f$ for every $1\leq i\leq k$ in the function system $\mathcal{F}$ then we use the notation $t(\mathcal{L},f)$ for $t(\mathcal{L},\mathcal{F})$.

In this chapter we address the following type of problem.

\medskip

{\it  Assume that $\mathcal{L}=\{L_1,L_2,\dots,L_k\}$ is a linear configuration and $\mathcal{A}$ is a class of compact abelian groups. Under what conditions on $\mathcal{L}$ and $\mathcal{A}$ is the function $f\mapsto t(\mathcal{L},f)$  continuous in the metric $d$ when functions are assumed to be uniformly bounded measurable functions on groups in $\mathcal{A}$ ?}

\medskip

The role of the class $\mathcal{A}$ is to exclude certain degeneracies that occur for number theoretic reasons. For example the linear form $2x$ becomes degenerated on the elementary abelian group $(\mathbb{Z}/2\mathbb{Z})^m$. We will need the following definition introduced by Gowers and Wolf in a slightly different form in \cite{GW}.

\begin{definition}\label{truecomp} Let $\mathcal{L}=\{L_1,L_2,\dots,L_k\}$ be a linear configuration. The {\bf true complexity} of $\mathcal{L}$ in a class $\mathcal{A}$ of abelian groups is the smallest number $m\in\mathbb{N}$ with the following property. For every $\epsilon>0$ there exists $\delta>0$ such that if $A\in\mathcal{A}$ is any abelian group and $\mathcal{F}=\{f_i\}_{i=1}^k$ is a system of measurable functions with $|f_i|\leq 1$ and $\|f_j\|_{U_{m+1}}\leq\delta$ for some $j$ then $t(\mathcal{L},\mathcal{F})\leq\epsilon$.
\end{definition}

In the above definition $\|.\|_{U_{m+1}}$ denotes Gowers's $m+1$-th uniformity norm.
Our main theorem states is the following.

\begin{theorem}\label{denscont} Let $a>0$. Let $\mathcal{L}$ be a linear configuration and $\mathcal{A}$ be a family of compact abelian groups such that $\mathcal{L}$ has true complexity at most $1$ in $\mathcal{A}$. Then $f\rightarrow t(\mathcal{L},f)$ is continuous with respect to the metric $d$ for measurable functions $f\in L^\infty(A)$ with $A\in\mathcal{A}$ and $|f|\leq a$.
\end{theorem}

\section{Ultra products and ultralimits}

Let $\omega$ be a non principal ultra filter on the natural numbers. Let $\{X_i\}_{i=1}^\infty$ be a sequnece of sets. For two elements $x=(x_1,x_2,\dots)$ and $y=(y_1,y_2,\dots)$ in the product $\prod_{i=1}^\infty X_i$ we say that $x\sim_\omega y$ if $\{i~|~x_i=y_i\}\in\omega$. It is well known that $\sim_\omega$ is an equivalence relation. The set $\prod_{\omega}X_i:=\bigl(\prod_{i=1}^\infty X_i\bigr)/\sim_\omega$ is called the {\it ultraproduct} of the sets $X_i$. 

Let $T$ be a compact Hausdorrf topological space and let $\{t_i\}_{i=1}^\infty$ be a sequence in $T$. The ultralimit $\lim_\omega t_i$ is the unique point $t$ in $T$ with the property that for every open set $U$ containing $t$ the set $\{i~|~t_i\in U\}$ is in $\omega$.
Let $\{f_i:X_i\rightarrow T\}_{i=1}^\infty$ be a sequence of functions. We define $f=\lim_\omega f_i$ as the function on $\prod_\omega X_i$ whose value on the equivalence class of $\{x_i\in X_i\}_{i=1}^\infty$ is $\lim_\omega f_i(x_i)$. 

Let $\{X_i,\mu_i\}_{i=1}^\infty$ be pairs where $X_i$ is a compact Hausdorff space and $\mu_i$ is a probability measure on the Borel sets of $X_i$.
We denote by $\bX$ the ultra product space $\prod_{\omega}X_i$.
The space $\bX$ has the following structures on it.

\medskip

\noindent{\it Strongly open sets:} ~We call a subset of $\bX$ strongly open if it is the ultra product of open sets $\{S_i\subset X_i\}_{i=1}^\infty$.

\medskip

\noindent{\it Open sets:}~We say that $S\subset \bX$ is open if it is a countable union of strongly open sets. Open sets on $\bX$ form a $\sigma$-topology. This is similar to a topology but it has the weaker axiom that only countable unions of open sets are required to be open. It can be proved that $\bX$ with this $\sigma$-topology is countably compact. This means that if $\bX$ is covered by countably many open sets then there is a finite sub-system which covers $\bX$.

\medskip

\noindent{\it Borel sets:} A subset of $\bX$ is called Borel if it is in the $\sigma$-algebra generated by strongly open sets. 

\medskip

\noindent{\it Ultra limit measure:}  If $S\subseteq \bX$ is a strongly open set of the form $S=\prod_\omega S_i$ then we define $\mu(S)$ as $\lim_\omega\mu_i(S_i)$. It is well known that $\mu$ extends as a probability measure to the $\sigma$-algebra of Borel sets on $\bX$. 

\medskip

\noindent{\it Ultra limit functions:} Let $T$ be a compact Hausdorff topological space. Let $\{f_i:X_i\rightarrow T\}_{i=1}^\infty$ be a sequence of Borel measurable functions. We call functions of the form $f=\lim_\omega f_i$ ultra limit functions. It is easy to see that ultra limit functions can always be modified on a $0$ measure set that they becomes measurable in the Borel $\sigma$-algebra on $\bX$. This means that ultra limit functions are automatically measurable in the completion of the Borel $\sigma$-algebra.

\medskip

\noindent{\it Measurable functions:} It is an important fact (see \cite{ESz}) that every bounded measurable function on $\bX$ is almost everywhere equal to some ultra limit function $f=\lim_\omega f_i$. 

\medskip

\noindent{\it Continuity:} A function $f:\bX\rightarrow T$ from $\bX$ to a topological space $T$ is called continuous if $f^{-1}(U)$ is open in $\bX$ for every open set in $T$. If $T$ is a compact Hausdorff topological space then $f$ is continuous if and only if it is the ultra limit of continuous functions $f_i:X_i\rightarrow T$. Furthermore the image of $\bX$ in a compact Hausdorff space $T$ under a continuous map is compact. 

\medskip

\section{The Fourier $\sigma$-algebra}

If $A$ is a compact Abelian group then linear characters are continuous homomrphisms of the form $\chi:A\rightarrow\mathcal{C}$ where $\mathcal{C}$ is the complex unit circle with multiplication as the group operation. Note that on compact abelian groups we typically use $+$ as the group operation. However if we think of $\mathcal{C}$ as a subset of $\mathbb{C}$ then we are forced to use multiplicativ notation. On the other hand, if we think of $\mathbb{C}$ as the group $\mathbb{R}/\mathbb{Z}$ then we are basically forced to use additive notation. 

Linear characters are forming the Fourier basis in $L^2(A)$. In particular linear characters generate the whole Borel $\sigma$-algebra on $A$. Assume now that $\bA=\prod_\omega A_i$ is the ultraproduct of compact abelian groups. Linear characters of $\bA$ can be similarly defined as for compact abelian groups. In this case we require them to be continuous in the $\sigma$-topology on $\bA$.

\begin{proposition}\label{rigidity} A function $\chi\in L^\infty(\bA)$ is a linear character if and only if $\chi=\lim_\omega \chi_i$ for some sequence $\{\chi_i\in L^\infty(A_i)\}_{i=1}^\infty$ of linear characters. 
\end{proposition}

The proof of the lemma relies on a rigidity result saying that almost linear characters on compact groups can be corrected to proper characters.

\begin{lemma}\label{rigid} For every $\epsilon>0$ there is $\delta>0$ such that if $f:A\rightarrow\mathbb{C}$ is a continuous function on a compact abelian group $A$ with the property that $|f(x+a)f^*(x)-f(y+a)f^*(y)|\leq\delta$ , $||f(x)|-1|\leq\delta$ for every $x,y,a\in A$ and $|f(0)-1|\leq\delta$ then there is a character $\chi$ of $A$ such that $|\chi(x)-f(x)|\leq\epsilon$ holds for every $x\in A$.
\end{lemma}

\begin{proof} As a tool we introduce group theoretic expected values of random variables taking values in $\mathcal{C}$. Let $l$ denote the arc length metric on the circle group $\mathcal{C}\simeq\mathbb{R}/\mathbb{Z}$ normalized by the total length $2\pi$. It is clear that the metric $l$ is topologically equivalent with the complex metric $|x-y|$ on $\mathcal{C}$.  Assume that a random variable $X$ takes its values in an arc of the circle group of length $1/3$. Then there is a lift $Y$ of $X$ to $\mathbb{R}$ such that $Y+\mathbb{Z}=X$ and $Y$ takes its values in an interval of length $1/3$. The lift $Y$ with this property is unique up to an integer shift. Then we define $\mathbb{E}(X)\in\mathbb{R}/\mathbb{Z}$ as $\mathbb{E}(Y)+\mathbb{Z}$. Switching to multiplicative notation in $\mathcal{C}$ this expected value satisfies $\mathbb{E}(X_1X_2)=\mathbb{E}(X_1)\mathbb{E}(X_2)$ where $X_1,X_2$ take values in an arc of length $1/6$.

Let us define $f_2(x)=f(x)/|f(x)|$. If $\delta<1$ then $f(x)\neq 0$ on $A$ and thus $f_2$ is defined on $A$. If $\delta>0$ is small enough then for every fixed $t$ the function $x\mapsto f(x+t)f^*(x)$ takes values in an arc of length at most $1/6$. For every $t\in A$ let $g(t)=\mathbb{E}_x(f(x+t)f^*(x))$ where $\mathbb{E}$ is the group theoretic expected value. If $\delta$ is small enough then $|g(t)-f(t)|\leq\epsilon$ holds for every $t\in A$ because $|f(x+t)f^*(x)-f(t)f^*(0)|\leq\delta$ and $f(0)$ is close to $1$.  Using our multiplicativity property of $\mathbb{E}$ we have for every pair $a,b\in A$ that
$$g(a+b)g^*(b)=\mathbb{E}_x(f(x+a+b)f^*(x)f^*(x+b)f(x))=\mathbb{E}_x(f(x+a+b)f^*(x+b))=$$
$$=\mathbb{E}_x((x+a)f^*(x))=g(a).$$ This implies that $g$ is a linear character of $A$.  
\end{proof}

Now we are ready to prove proposition \ref{rigidity}

\begin{proof}     The continuity of $\chi$ guarantees that $\chi=\lim_\omega f_i$ for some sequence of continuous functions $f_i$ on $A_i$. The fact that $\chi$ is a character implies that there is a sequence $\delta_i$ such that $f_i$ satisfies the conditions of lemma \ref{rigid} with $\delta_i$ for every $i$ and $\lim_\omega\delta_i=0$. It follows by lemma \ref{rigid} that there is a sequence of linear characters $\chi_i$ on $A_i$ such that $\lim_\omega\max(|\chi_i-f_i|)=0$. Thus we have that $\lim_\omega \chi_i=\lim_\omega f_i=\chi$.
\end{proof}

\medskip

Proposition \ref{rigidity} implies that the set of linear characters of $\bA$ (also as a group) is equal to $\prod_\omega\hat{A_i}$. We denote this set by $\hat{\bA}$. If $f\in L^2(\bA)$ then the Fourier transform of $f$ on $\bA$ is the function $\hat{f}\in l^2(\hat{\bA})$ defined by $\hat{f}(\chi)=(f,\chi)$. If $f=\lim_\omega f_i$ then we have that $\hat{f}=\lim_\omega \hat{f_i}$. 

 It was observed in \cite{Sz1} that linear characters of $\bA$ no longer span $L^2(\bA)$. This shows that in general we only have $\|\hat{f}\|_2\leq\|f\|_2$ instead of equality.  Furthermore the $\sigma$-algebra $\mathcal{F}(\bA)$ generated by linear characters on $\bA$ is smaller than the whole ultraproduct $\sigma$-algebra on $\bA$. (The only exception is the case when $\bA$ is a finite group. This can happen if the groups $A_i$ are finite and there is a uniform bound on their size.)

We call $\mathcal{F}(\bA)$ the {\bf Fourier $\sigma$-algebra} on $\bA$. The fact that the Fourier $\sigma$-algebra is not the complete $\sigma$-algebra on $\bA$ gives rise to the interesting operation $f\mapsto\mathbb{E}(f|\mathcal{F}(\bA))$ that isolates the ``Fourier part'' of a function $f\in L^2(\bA)$. Using that linear characters of $\bA$ are closed with respect to multiplication we obtain that linear characters are forming a basis in $L^2(\mathcal{F}(\bA))$. This implies that if $f\in L^2(\bA)$ then $\hat{f}=\hat{g}$ where $g=\mathbb{E}(f|\mathcal{F}(\bA))$. Thus we have that $\|\hat{f}\|_2=\|\hat{g}\|_2=\|\mathbb{E}(f|\mathcal{F}(\bA))\|_2$. In particular $\|f\|_2=\|\hat{f}\|_2$ holds if and only if $f$ is measurable in $\mathcal{F}(\bA)$.

The Fourier $\sigma$-algebra has an elegant description in terms of the second Gowers norm $U_2$. Recall that the $U_2$ norm \cite{Gow},\cite{Gow2} of a function $f\in L^\infty(A)$ on a compact abelian group $A$ is defined by
\begin{equation}\label{u2}
\|f\|_{U_2}=\Bigl(\mathbb{E}_{x,a,b\in A} f(x)f(x+a)^*f(x+b)^*f(x+a+b)\Bigr)^{1/4}.
\end{equation}

The next lemma gives a description of the $U_2$-norm in terms of Fourier analysis.

\begin{lemma}\label{u2eq} If $f\in L^\infty(A)$ then $\|f\|_{U_2}=\|\hat{f}\|_4$ and thus $\|\hat{f}\|_\infty\leq \|f\|_{U_2}\leq (\|f\|_2\|\hat{f}\|_\infty)^{1/2}$. 
\end{lemma}

One can define $\|f\|_{U_2}$ by the formula (\ref{u2}) for functions on ultraproduct groups. With this definition we have that $\|f\|_{U_2}=\lim_{\omega}\|f_i\|_{U_2}$ whenever $f=\lim_\omega f_i$. The main differnece from the compact case is that $\|.\|_{U_2}$ is no longer a norm for functions in $L^\infty(\bA)$. It is only a semi-norm. However the next lemma shows that $\|.\|_{U_2}$ is a norm when restricted to $L^\infty(\mathcal{F}(\bA))$ and that $\mathcal{F}(\bA)$ is the largest $\sigma$-algebra with this property.

\begin{lemma} If $g\in L^\infty(\bA)$ then $\|g\|_{U_2}=0$ if and only if $g$ is orthogonal to $L^2(\mathcal{F}(\bA))$. A function $f\in L^\infty(\bA)$ is measurable in $\mathcal{F}(\bA)$ if and only if $f$ is orthogonal to every function $g\in L^\infty(\bA)$ with $\|g\|_{U_2}=0$. In particular we have that $\|.\|_{U_2}$ is a norm on $L^\infty(\mathcal{F}(\bA))$.
\end{lemma}

\begin{proof} We can assume that $g=\lim_\omega g_i$ for some sequence of functions $\{g_i\in L^\infty(A_i)\}_{i=1}^\infty$ such that $\|g_i\|_\infty\leq\|g\|_\infty$ holds for every $i$.
Assume first that $\|g\|_{U_2}=0$. Let $\chi=\lim_\omega\chi_i$ be an ultralimit of linear characters. Using lemma \ref{u2eq} we have that $|(g_i,\chi_i)|\leq\|\hat{g}_i\|_\infty\leq\|g_i\|_{U_2}$ and thus $$|(g,\chi)|=\lim_\omega |(g_i,\chi_i)|\leq\lim_\omega\|g_i\|_{U_2}=\|g\|_{U_2}=0.$$ It follows that $g$ is orthogonal to the space $L^2(\mathcal{F}(\bA))$ spanned by linear characters of $\bA$. For the other direction assume that $g\neq 0$ is orthogonal $L^2(\mathcal{F}(\bA))$. For every $i$ we choose a linear character $\chi_i$ on $A_i$ such that $|(g_i,\chi_i)|=\|\hat{g}_i\|_\infty$. We have by lemma \ref{u2eq} and by $\|g_i\|_2\leq\|g_i\|_\infty\leq\|g\|_\infty$ that $|(g_i,\chi_i)|\geq\|g_i\|_{U_2}^2\|g\|_\infty^{-1}$. Then we have for $\chi=\lim_\omega \chi_i$ that $0=|(g,\chi)|\geq (\lim_\omega\|g_i\|_{U_2}^2)\|g\|_\infty^{-1}$. It follows that $\|g\|_{U_2}=0$.

To complete the proof assume that $f\in L^\infty(\bA)$ is orthogonal to every $g\in L^\infty(\bA)$ with $\|g\|_{U_2}=0$. Let $g:=f-\mathbb{E}(f|\mathcal{F}(\bA))\in L^\infty(\bA)$. Note that since $\mathbb{E}$ is an orthogonal projection it follows that $(f,g)=\|g\|_2^2$. We have that $g$ is orthogonal to $L^2(\mathcal{F}(\bA))$ and so $\|g\|_{U_2}=0$. It implies that $(f,g)=0$ but that is only possible if $g=0$ and $f=\mathbb{E}(f|\mathcal{F}(\bA))$.
\end{proof}

\medskip

Let $\hat{\mathcal{Q}}:L^2(\bA)\rightarrow\mathcal{M}$ be such that $\hat{\mathcal{Q}}(f)$ is the isomorphism class of $\hat{f}$ in $\mathcal{M}$.
Let furthermore $\mathcal{Q}(f)$ denote the isomorphism class in $\mathcal{H}$ representing the Fourier transform of $\hat{\mathcal{Q}}(f)$. Note that $\mathcal{Q}(f)=\mathcal{Q}(\mathbb{E}(f|\mathcal{F}(\bA)))$. We have that $\mathcal{Q}(f)$ can be represented as a measurable function on some second countable compact abalian group with $\|\mathcal{Q}(f)\|_2\leq\|f\|_2$ which in some sense imitates $f$. However it is not even clear from this definition that if $f$ is a bounded function then $\mathcal{Q}(f)$ is also bounded. The next theorem provides a structure theorem for functions in $L^\infty(\mathcal{F}(\bA))$ and describes $\mathcal{Q}(f)$.

\begin{theorem}\label{factor} A function $f\in L^\infty(\bA)$ is measurable in $\mathcal{F}(\bA)$ if and only if there is a continuous, surjective, measure preserving homomorphism $\phi:\bA\rightarrow A$  to some second countable compact abelian group $A$ and a function $h\in L^\infty(A)$ such that $f=h\circ\phi$ (up to $0$ measure change). Furthermore $d(h,\mathcal{Q}(f))=0$ implying that the isomorphism class of $h$ is $\mathcal{Q}(f)$.
\end{theorem} 

\begin{proof} Assume first that $f=h\circ\phi$ for some homomorphism $\phi$ and function $h$ as in the statement. 
Let $h=\sum_{i=1}^\infty\lambda_i\chi_i$ be the Fourier decomposition of $h$ converging in $L^2(A)$ where $\chi_i$ is a sequence of linear characters of $A$. We have that $\chi_i\circ\phi$ is a linear character of $\bA$ for every $i$. The measure preserving property of $\phi$ implies that $f=\sum_{i=1}^\infty\lambda_i(\chi_i\circ\phi)$ and thus $f$ is measurable in $\mathcal{F}(\bA)$. 

For the other direction assume that $f\in L^\infty(\mathcal{F}(\bA))$. Then $f=\sum_{i=1}^\infty a_i\chi_i$ for some (distinct) linear characters $\{\chi_i\}_{i=1}^\infty$ of $\bA$ where the convergence is in $L^2$ and $\|f\|_2^2=\sum_{i=1}^\infty |a_i|^2$.
Let us consider the homomorphism $\phi:\bA\rightarrow\mathcal{C}^\infty$ such that the $i$-th coordinate of $\phi(x)=\chi_i(x)$. 
Using the continuity of $\phi$ we have that the image $A$ of $\phi$ is a closed subgroup in $\mathcal{C}^\infty$. 
Let $\nu$ denote the Borel measure on $A$ satisfying $\nu(S)=\mu(\phi^{-1}(S))$ where $\mu$ is the ultralimit measure on $\bA$. The fact that $\phi$ is a homomorphism implies that $\nu$ is a group invariant Borel probability measure on $A$ and thus $\nu$ is equal to the normalized Haar measure. In other words $\phi$ is measure preserving with respect to the Haar measure on $A$.  

Let us denote by $\alpha_i$ the $i$-th coordinate function on $A$. It is clear that $\{\alpha_i\}_{i=1}^\infty$ is a system of linear characters of $A$. Since $\phi$ is surjective it induces an injective homomorphism $\hat{\phi}:\hat{A}\rightarrow\hat{\bA}$ defined by $\hat{\phi}(\chi)=\chi\circ \phi$ with the property that $\hat{\phi}(\alpha_i)=\chi_i$ holds for every $i$.  We have that $h=\sum_{i=1}^\infty a_i\alpha_i$ (which is defined up to a $0$ measure set on $A$) is convergent in $L^2$ and has the property that $f=h\circ \phi$ (up to a $0$ measure set). The fact that $\hat{\phi}$ is an injective homomorphism implies that $\hat{d}(\hat{h},\hat{f})=0$ and thus $d(h,\mathcal{Q}(f))=0$.
\end{proof}

If $\mathcal{L}$ is a system of linear forms and $f\in L^\infty(\bA)$ then we can define $t(\mathcal{L},f)$ by the formula (\ref{confdens}) using the ultralimit measure on $\bA$.

\begin{proposition}\label{counting} Let $f\in L^\infty(\mathcal{F}(\bA))$ and let $\mathcal{L}$ be a system of linear forms. Then $t(\mathcal{L},f)=t(\mathcal{L},\mathcal{Q}(f))$. Furthermore if $\mathcal{L}$ has complexity $1$ in a family $\mathcal{A}$ of compact abelian groups, $\bA$ is an ultraproduct of groups in $\mathcal{A}$ and $f\in L^\infty(\bA)$ then $t(\mathcal{L},f)=t(\mathcal{L},\mathcal{Q}(f))$.
\end{proposition}

\begin{proof} For the first part we use theorem \ref{factor}. We get that $f=h\circ\phi$ for some measure preserving homomorphsim $\phi:\bA\rightarrow A$. It follows that $t(\mathcal{L},f)=t(\mathcal{L},h)=t(\mathcal{L},\mathcal{Q}(f))$. For the sencond part let $f=\lim_\omega f_i$ and $g=\mathbb{E}(f|\mathcal{F}(\bA))=\lim_\omega g_i$ for some functions with $\|f_i\|_\infty\leq\|f\|_\infty$ and $\|g_i\|_\infty\leq\|g\|_\infty$. We have that $\lim_\omega\|f_i-g_i\|_{U_2}=\|f-g\|_{U_2}=0$. Then using that $\mathcal{L}$ has complexity $1$ we obtain $t(\mathcal{L},\mathcal{Q}(f))=t(\mathcal{L},\mathcal{Q}(g))=t(\mathcal{L},g)=\lim_\omega t(\mathcal{L},g_i)=\lim_\omega t(\mathcal{L},f_i)=t(\mathcal{L},f)$.
\end{proof}

\section{The ultraproduct descriptions of $\hat{d}$ and $d$ convergence}

We give a simple and useful description of $\hat{d}$-convergence using ultrafilters. The price that we pay for the simplicity is that we don't get an explicit metric on $\mathcal{M}$, we only get the concept of convergence.

\begin{theorem}\label{charconv1} Let $a>0$. Assume that $\{f_i\}_{i=1}^\infty$ is a sequence in $\mathcal{M}_a$ that converges to $f$ in $\hat{d}$ then $f$ is isomorhic to $\lim_\omega f_i$ for every (non-principal) ultrafilter $\omega$. Consequently a sequence $\{f_i\}_{i=1}^\infty$ in $\mathcal{M}_a$ is convergent in $\hat{d}$ if and only if the isomorphism class of $\lim_\omega f_i$-limit doesn't depend on the choice of the ultra filter $\omega$.
\end{theorem}

\begin{proof} For every $i$ let $\alpha_i:T_i\rightarrow S_i$ be an $\epsilon_i$-isomorphism between $f_i$ and $f$ with  $T_i\subseteq G_i, S_i\subseteq G$ such that $\lim_{i\to\infty}\epsilon_i=0$. 
Assume that $\{h_i\}_{i=1}^\infty$ represents an element $h$ in $\prod_\omega G_i$ that is in $\supp(g)$ where $g=\lim_\omega f_i$. We have for some set $S\in\omega$ that $f_i(h_i)>g(h)/2$ and $\epsilon_i\leq g(h)/4$ for $i\in S$. It follows that $\alpha_i(h_i)\in\supp_{g(h)/4}(f)$ holds for every $i\in S$. Since $\supp_{g(h)/4}$ is finite we have that $\lim_\omega \alpha_i(h_i)$ exists and it is an element in $G$ that we denote by $\beta(h)$. The map $\beta:\supp(g)\rightarrow\supp(f)$ is a partial isomorphis of arbitrary high weight and so it extends to an isomorphism from $\langle g\rangle$ to $\langle f\rangle$.  It is clear that $\beta$ is also an isomorphism between $f$ and $g$.  
\end{proof}

\begin{corollary}\label{charconv2} Let $a>0$. Assume that $\{f_i\}_{i=1}^\infty$ is a sequence of functions with $f_i\in L^\infty(A_i)$ and $\|f_i\|_\infty\leq a$ for some sequence $\{A_i\}_{i=1}^\infty$ of compact abelian groups. If $\{f_i\}_{i=1}^\infty$ converges to $f\in\mathcal{H}_a$ in the metric $d$ then $f=\mathcal{Q}(\lim_\omega f_i)$ for an arbitrary (non-principal) ultrafilter $\omega$. 
\end{corollary}

\begin{proof} Since the Fourier transform of $f'=\lim_\omega f_i$ is the ultra limit of the Fourier transforms of $f_i$ we have by theorem \ref{charconv1} that $\hat{d}(\hat{f'},\hat{f})=0$. It follows that $\mathcal{Q}(f')=f$.  
\end{proof}

\begin{corollary}\label{charconv3} Let $a>0$. Assume that $\{f_i\}_{i=1}^\infty$ is a convergent sequence of functions with $f_i\in L^\infty(A_i)$ and $\|f_i\|_\infty\leq a$ for some sequence $\{A_i\}_{i=1}^\infty$ of compact abelian groups. Then the limit $f$ of $\{f_i\}_{i=1}^\infty$ can be represented as a function $f:A\in L^\infty(A)$ where the dual group of $A$ is a subgroup in $\prod_{\omega}\hat{A_i}$.
\end{corollary}

\begin{proof} We have by corollary \ref{charconv2} that $f=\mathcal{Q}(\lim_\omega f_i)$. This means that $\hat{f}$ has an injective embedding into $\hat{\bA}$ where $\bA=\prod_\omega A_i$. By $\hat{\bA}=\prod_{\omega}\hat{A_i}$ the proof is complete.
\end{proof}

Corollary \ref{charconv3} gives a useful restriction on the structure of the group on which the limit function of a convergent seqence is defined. For example if $A_i$ are growing groups of prime order then the limit function is defined on a compact group whose dual group is torsion-free. On the other hand if $p$ is a fix prime and $f_i$ is defined on $\mathbb{Z}_p^i$ then the limit function is defined on the compact group $\mathbb{Z}_p^\infty$.

\section{Proofs of theorems \ref{convcl}, \ref{weakcl}, \ref{denscont} }

For the proofs of theorem \ref{convcl} and theorem \ref{weakcl} assume that $\{f_i\}_{i=1}^\infty$ is a convergent sequence in $\mathcal{H}(K)$ for some convex compact set $K\subseteq\mathbb{C}$. Corollary \ref{charconv2} implies that the limit is $\mathcal{Q}(f)$ where $f=\lim_\omega f_i$. Note that $f$ takes its values in $K$. We have that $\mathcal{Q}(f)=\mathcal{Q}(g)$ where $g=\mathbb{E}(f|\mathcal{F}(\bA))$. It follows by theorem \ref{factor} that $g=h\circ\phi$ for some measure preserving homomorphism $\phi:\bA\rightarrow A$ and the isomorphism class of $h$ is $\mathcal{Q}(g)$. Since $g$ is a projection of $f$ to a $\sigma$-algebra we have that $g$ (and thus $h$) takes its values in $K$. This completes the proof of theorem \ref{convcl}.

For the proof of theorem \ref{weakcl} assume that $f_i$ is tightly convergent and $K=\{x:x\in\mathbb{C},\|x\|\leq a\}$. Then, using the above notation we have that $\|g\|_2=\|h\|_2=\lim_{i\to\infty}\|f_i\|_2=\lim_\omega\|f_i\|_2=\|f\|_2$ where we use tightness in the second equality. This is only possibel if $f=g$ and thus $\mu_h=\mu_f=\lim_\omega \mu_{f_i}$ holds. Since this is true for every ultrafilter $\omega$ we obtain that $\lim_{i\to\infty}\mu_{f_i}=\mu_h$ holds with respect to weak convergence of measures.

To prove theorem \ref{denscont} assume that $\mathcal{L}$ has complexity $1$ and $f_i$ is a $d$ convergent sequence as above. Using the above notation and proposition \ref{counting} we have that $\lim_\omega t(\mathcal{L},f_i)=t(\mathcal{L},f)=t(\mathcal{L},\mathcal{Q}(f))$ where (using corollary \ref{charconv2}) $\mathcal{Q}(f)$ is equal to the $d$-limit of the sequence $\{f_i\}_{i=1}^\infty$. Since this is true for every ultrafilter $\omega$ the proof is complete.

\section{Proof of theorem \ref{rothext}}

For the proof of theorem \ref{rothext} we will need the next proposition which is interesting on its own right.

\begin{proposition}\label{visszahuz} Let $B$ be a compact abelian group with torsion-free dual group and let $f:B\rightarrow [0,1]$ be an arbitrary measurable function. Then there are subsets $S_p\subseteq\mathbb{Z}_p$ for every prime number $p$ such that the functions $1_{S_p}$ converge to $f$.
\end{proposition}

\begin{lemma}\label{discretize} For every $\epsilon$ there is $N(\epsilon)$ such that if $A$ is a finite abelian group with $|A|\geq N(\epsilon)$ and $f:A\rightarrow [0,1]$ is a function then there is a function $h:A\rightarrow\{0,1\}$ such that $\|f-h\|_{U_2}\leq\epsilon$. 
\end{lemma}

\begin{proof}  Let us fix $\epsilon>0$. Let $f:A\rightarrow [0,1]$ be a function on a finite abelian group. Let $h$ be the random function on $A$ whose distribution is uniquely determined by the following properties: 1.) $h$ is $\{0,1\}$-valued, 2.) $\{h(a)~|~a\in A\}$ is an independent system of random variables and 3.) $\mathbb{E}(h(a))=f(a)$ holds for every $a\in A$. We claim that with a large probability the function $h-f$ has $U_2$ norm at most $\epsilon$ if $|A|$ is big enough. Obsereve that $X_a:=h(a)-f(a)$ is a random variable for each $a\in A$ with $0$ expectation and $\|X_a\|_\infty\leq 1$. The random variables $X_a$ are all independent. Let $\chi:A\rightarrow\mathbb{C}$ be a linear character. Then we have that $(h-f,\chi)=|A|^{-1}\sum_{a\in A} X_a\chi(a)$.  By Chernoff's bound we have that $\mathbb{P}(|(h-f,\chi)|\geq\epsilon^2)$ is exponentially small in $|A|$. This implies that if $|A|$ is large enough then with probability close to $1$ we have that $\|\hat{h}-\hat{g}\|_\infty\leq\epsilon^2$ and thus by lemma  \ref{u2eq} we get $\|h-g\|_{U_2}\leq\epsilon$ holds in these cases.  
\end{proof}

\medskip

\noindent{\it Proof of proposition \ref{visszahuz}.}~~ For a number $n$ let $a(n)$ denote the minimum of $d(1_S,f)$ where $S$ is a subset in $\mathbb{Z}_n$. The statement of the proposition is equivalent with $\lim_{p\to\infty} a(n)=0$ where $p$ runs through the prime numbers. Assume by contradiction that there is $\epsilon>0$ and a growing infinite sequence $\{p_i\}_{i=1}^\infty$ of prime numbers with $a(p_i)>\epsilon$. Let $A_i=\mathbb{Z}_{p_i}$ and $\bA=\prod_\omega A_i$. We have that $\hat{\bA}=\prod_\omega\hat{A_i}\simeq\prod_\omega A_i=\bA$. Since $\bA$ is not only an abelian group but a field of $0$ characteristic with uncountable many elements we have that $\bA$ (and thus $\hat{\bA}$) as an abelian group is isomorphic to an infinite direct sum of $\mathbb{Q}^+$. It follows that the torsion-free group $\hat{B}$ has an embedding $\hat{\phi}:\hat{B}\rightarrow\hat{\bA}$  into $\hat{\bA}$.
This embedding induces a continuous homomorphsim $\phi:\bA\rightarrow B$ in the way that $\phi(x)$ denotes the unique element in $B$ such that $\chi(\phi(x))=\hat{\phi}(\chi)(x)$ holds for every $\chi\in\hat{B}$.

Let $g=f\circ\phi$. We have that $g:\bA\rightarrow [0,1]$ is a measurable function and thus $g=\lim_\omega g_i$ for a system of functions $\{g_i:A_i\rightarrow [0,1]\}_{i=1}^\infty$. By lemma \ref{discretize} for every $i$ we can find a $0-1$ valued function $g'_i$ such that $\lim_{i\to\infty}\|g'_i-g_i\|_{U_2}=0$. By choosing a subsequence we can assume that both $\{g_i'\}_{i=1}^\infty$ and $\{g_i\}_{i=1}^\infty$ are $d$-convergent. Let $g'=\lim_\omega g_i'$. We have that $\|g-g'\|_{U_2}=0$ and thus since $g$ is measurable in $\mathcal{F}(\bA)$ we have that $g=\mathbb{E}(g'|\mathcal{F}(\bA))$. By corollary \ref{charconv2} we obtain that the $d$ limit of $\{g_i'\}_{i=1}^\infty$ is $f$. This implies that $0=\lim d(g_i',f)\geq \liminf a(p_i)\geq\epsilon$ which is a contradiction.

\medskip

Now we are ready to prove theorem \ref{rothext}. First observe that in Proposition \ref{visszahuz} we can assume with no additional cost that the sets $S_p$ have density at least $\mathbb{E}(f)$. This follows from the fact that their densities converge to $\mathbb{E}(f)$ and so it is enough to set a few values to $1$ (with density tending to $0$).
This observation together with Proposition \ref{visszahuz} and theorem \ref{confdens} imply that if $f:A\rightarrow [0,1]$ is a measurable function with $\mathbb{E}(f)=\delta$ on an abelian group with torsion-free dual then  $\rho(\delta,\mathcal{L})\leq t(\mathcal{L},f)$. It remains to find a function where equality holds. For every $p$ prime let $S_p\subseteq\mathbb{Z}_p$ be such that $|S_p|/p\geq\delta$ and that $t(\mathcal{L},1_{S_p})$ is minimal possible.
We can choose a $d$-convergent subsequence $\{f_i\}_{i=1}^\infty$ from $1_{S_p}$ such that $\lim_{i\to\infty} t(\mathcal{L},f_i)=\rho(\delta,\mathcal{L})$. Let $f$ be the limit of $\{f_i\}_{i=1}^\infty$. By theorem \ref{confdens} we have that $t(\mathcal{L},f)=\lim_{i\to\infty} t(\mathcal{L},f_i)=\rho(\delta,\mathcal{L})$. Corollary \ref{charconv3} guarantess that $f$ is defined on a group whose dual is torsion-free.

\section{Connection to dense graph limit theory and concluding remarks}\label{graphlim}

Let $H$ and $G$ be finite graphs. The density of $H$ in $G$ is the probability that a random map from $V(H)$ to $V(G)$ takes edges to edges. We denote this quantity by $t(H,G)$. One can generalize this notion of density for the case when $G$ is replaced by a symmetric bounded measurable function $W:\Omega^2\rightarrow\mathbb{C}$ where $(\Omega,\mu)$ is a probability space. Then $t(H,W)$ is defined by $$t(H,W):=\int_{x_1,x_2,\dots,x_n\in\Omega}~\prod_{(i,j)\in E(H)}W(x_i,x_j)~~d\mu^n$$ where the verices of $H$ are indexed by $\{1,2,\dots,n\}$. It is easy to check that if $\Omega=V(G)$ , $\mu$ is the uniform distribution on $V(G)$ and $W:V(G)^2\rightarrow\{0,1\}$ is the adjacency matrix of $G$ then $t(H,G)=t(H,W)$. 

In the framework of dense graph limit theory, a sequence of graphs $\{G_i\}_{i=1}^\infty$ is called convergent if for every fixed graph $H$ the sequence $\{t(H,G_i)\}_{i=1}^\infty$ is convergent. It was proved in \cite{LSz} that for a convergent graph sequence $\{G_i\}_{i=1}^\infty$ there is a limit object of the form of a symmetric measurable function $W:\Omega^2\rightarrow [0,1]$ (called a {\it graphon}) such that for every graph $H$ we have $\lim_{i\to\infty} t(H,G_i)=t(H,W)$. 

In the above theorem $\Omega$ can be chosen to be $[0,1]$ with the uniform measure however in many cases it is more natural to use other probability spaces. We investigate the case when $(\Omega,\mu)$ is a compact abelian group $A$ with the normalized Haar measure. Let $f:A\rightarrow\mathbb{C}$ be a bounded measurable function and let $W_f:A^2\rightarrow\mathbb{C}$ be defined by $W_f(x,y):=f(x+y)$. As it was pointed out in the introduction, for a finite graph $H$ the density $t(H,W_f)$ is equal to $t(\mathcal{L},f)$ where $\mathcal{L}_H:=\{x_i+x_j~:~(i,j)\in E(H)\}$. Using this correspondence and our results in this paper we get the following theorem on graph limits.

\begin{theorem}\label{tgraphlim} Let $\{f_i:A_i\rightarrow K\}_{i=1}^\infty$ be a sequence of measurable functions on compact abelian groups with values in a compact convex set $K\subseteq\mathbb{C}$. Assume that $\lim_{i\to\infty}t(H,W_{f_i})$ exists for every graph $H$. Then there is a measurable function $f:A\rightarrow K$ on a compact abelian group $A$ such that $\lim_{i\to\infty} t(H,W_{f_i})=t(H,W_f)$ holds for every graph $H$.
\end{theorem}

\begin{proof} By chosing a subsequence we can assume by theorem \ref{convcl} that $\{f_i\}_{i=1}^\infty$ is convergent in $d$ with limit $f:A\rightarrow K$. Then by theorem \ref{confdens} we obtain that $\lim_{i\to\infty}t(\mathcal{L}_H,f_i)=t(\mathcal{L}_H,f)$ holds for every graph $H$. This completes the proof.
\end{proof}

Theorem \ref{tgraphlim} is closely related to the results in \cite{LSz3}. Let $f:G\rightarrow [0,1]$ be a mesurable function on a compact but not necessarily commutative group. Assume that the technical condition $f(g)=f(g^{-1})$ holds for every $g\in G$. Then the function $W:G^2\rightarrow [0,1]$ defined by $W(x,y)=f(xy^{-1})$ is symmetric. We call graphons of this type Cayley grphons. It was proved in \cite{LSz3} that limits of Cayley graphons are also Cayley graphons. 
This theorem implies that one can talk about limits of functions on compact topological groups and the limit object are also functions on compact topological groups.
More complicated limit objects come into picture if in the commutative setting when we wish for the continuity of densities of linear configurations of higher complexity. As it was showed in \cite{Sz}, this refinement of the limit concept requires more complicated limit objects. There are examples for functions on abelian groups converging to functions on nilmanifolds.

\medskip

\noindent{\bf Acknowledgment}~~  {\it This research was supported by the European Research Council, project: Limits of discrete structures, 617747}

\end{document}